\documentclass[letterpaper,10pt,conference]{ieeeconf}

\IEEEoverridecommandlockouts

\usepackage{subcaption}
\usepackage{mathtools}
\usepackage{amsfonts}
\usepackage{hyperref}
\usepackage[usenames,dvipsnames]{xcolor}
\usepackage{cite} 

\def\tp{\mathsf{T}}

\renewcommand{\u}{\boldsymbol{u}}
\newcommand{\x}{\boldsymbol{x}}
\newcommand{\w}{\boldsymbol{w}}

\newcommand{\R}{\mathbb{R}}

\usepackage{comment}

\newif\ifshow
\showtrue
\ifshow

\else
\excludecomment{scratch}
\fi

\newtheorem{remark}{\bfseries Remark}
\newtheorem{example}{\bfseries Example}

\newtheorem{theorem}{\bfseries Theorem}

\title{Stochastic optimal control using semidefinite programming \\ for moment dynamics}
\author{Andrew Lamperski, Khem Raj Ghusinga, and  Abhyudai Singh
\thanks{ Andrew Lamperski  is with the Department of Electrical and Computer Engineering, University of Minnesota, Minneapolis, MN, USA 55455.
        {\tt\small alampers@umn.edu}}%
\thanks{Khem Raj Ghusinga is with the Department of Electrical and Computer Engineering, University of Delaware, Newark, DE, USA 19716.
        {\tt\small khem@udel.edu}}%
\thanks{Abhyudai Singh  is with the Faculty of Electrical and Computer Engineering, University of Delaware, Newark, DE, USA 19716.
        {\tt\small absingh@udel.edu}}
}

\begin{document}
\maketitle

\begin{abstract}
This paper presents a method to approximately solve stochastic optimal
control problems in which the cost function and the system dynamics
are polynomial. For stochastic systems with polynomial dynamics, the
moments of the state can be expressed as a, possibly infinite, system
of deterministic linear
ordinary differential equations.   
By casting the problem as a deterministic control
problem in moment space, semidefinite programming is used to find a
lower bound on the optimal solution. 
The constraints in the
semidefinite program are imposed by the ordinary differential
equations for moment dynamics and semidefiniteness of the outer
product of moments. 
From the solution to the semidefinite program, an approximate optimal
control strategy can be constructed using a least squares method. 
In the linear quadratic case, the method gives an exact solution to
the optimal control problem.
In more complex problems, an infinite number of moment differential
equations would be required to compute the optimal control law. In
this case, we give a procedure to increase the size of the
semidefinite program, leading to increasingly accurate approximations to
the true optimal control strategy. 
\end{abstract} 

\section{Introduction}
Stochastic optimal control problems frequently arise in a variety of settings such as engineering, management, finance, ecology, etc. where a cost function is minimized by choosing the inputs to a stochastic differential equation \cite{flemingcontrolled2006}. The solution to a stochastic optimal control problem is typically formulated using the value function approach \cite{YongZhou99}. It turns out that aside from a few special cases such as linear quadratic control problems, it is generally not possible to find explicit solutions. For other problems, numerical methods are employed which require discretization of the state-space and time, incur significant computational cost, and generally suffer from the curse of dimensionality. Another prevalent approach to study nonlinear stochastic control problems has been to linearize the system to be able to apply the linear stochastic control methods\cite{youngoptimal1988,jumarie1996improvement,Socha05}. However, in these methods it is difficult to get solutions with satisfactory accuracy as the linearization is only valid for small deviations.

Alternatively, one can transform the problem to moment-space wherein the cost function to be optimized is viewed as a function of moments of the state and control variables. Particularly when the cost function and the system dynamics are polynomial, the moment dynamics of the state is well characterized by linear ordinary differential equations, which allows one to study the problem as a deterministic optimal control problem. In \cite{jumarie1995practical,jumarie1996}, this approach has been used for few examples considering the system dynamics to be linear. In general, when the system dynamics is nonlinear, the system of differential equations for moment dynamics become infinite dimensional.

Here,  we provide a method to approximate the solution to stochastic optimal control problems wherein the system dynamics, and cost are polynomials. A lower bound for the optimal value of the cost function is computed using a semidefinite programming approach. Moreover, a polynomial control policy with time-varying coefficients can be extracted from the semidefinite program, using which an upper bound on the optimal value of the cost function can be computed via Monte Carlo simulations. The idea here is that if the difference between the upper and lower bounds is small, the controller extracted from the semidefinite program must be close to the optimal controller. Whereas for the systems for which moment dynamics is finite dimensional, the lower bound computed is unique; one can obtain an increasing sequence of lower bounds for a system with infinite-dimensional moment dynamics by increasing the number of moment equations and the number of constraints on the moment dynamics in the semidefinite program. For details on semidefinite programming, see \cite{cominetti2012modern, boyd2004convex}.

The paper is structured as follows. Section~\ref{sec:problem} discusses the problem statement, introduces some examples that we use to illustrate the method proposed in the paper, and briefly describes the results. In Section~\ref{sec:momentdynamics}, a system of ordinary differential equations characterizing the moment dynamics is derived for a stochastic differential equation. Section~\ref{sec:results} presents the main results of the paper. These results are illustrated via examples in Section~\ref{sec:examples}. In section~\ref{sec:conclusion}, conclusion of the paper and some directions of future work are given.
\subsection*{Notation}

Random variables will be denoted in bold: $\x$, and $\u$. Non-random
variables will be non-bold. For example, a specific value taken by
$\x$ would be denoted by $x$. 
The expected value of a random variable, $\boldsymbol{x}$ is denoted by $\left<\boldsymbol{x}\right>$.

\section{Problem}
\label{sec:problem}
\subsection{Setup}

This paper will provide a method to approximate the solution to scalar polynomial optimal control problems of the form:

\begin{subequations}
\label{eq:polyProb}
\begin{align} \label{eq:polyCost}
& \underset{\u}{\textrm{minimize}} && \left<\int_0^T c(\x_t,\u_t)dt +
                       h(\x(T))\right> \\
& \textrm{subject to} && 
d\x_t = f(\x_t,\u_t) dt + g(\x_t,\u_t) d\w_t \\
\label{eq:polyConstraint}
&&& b_l(\x_t,\u_t) \ge 0\quad \textrm{ for } l=1,\ldots,B \\
&&& \x(0) = x_0.
\end{align}
\end{subequations}
Here $\x_t\in \R$ is the state and $\u_t\in\R$ is the control;
$f(x,u): \R\times \R  \rightarrow \R$ and $g(x,u): \R\times\R
\rightarrow \R $ are polynomials that describe the system dynamics; 
$b_l(x,u): \R \times \R \rightarrow \R^M$ are polynomials
describing constraints; 
and $\w_t$ is a  Weiner process satisfying
\begin{align}
\left<d \w_t \right>=0, 
\quad
\left<d\w_t \, d\w_t^\top\right>= \,dt.
\end{align}
The control input $\u(t)$ is minimized over all Borel measureable
functions that are adapted to the filtration generated by
$\w(t)$. Typically, there is no loss of generality in restricting the
search to \emph{Markov} control policies, i.e. policies of the form,
$\u_t = \gamma(\x_t,t)$, where 
$\gamma$ is a Borel-measurable function. For details, see
\cite{flemingcontrolled2006}.

Since it is assumed that all of the functions, $c$, $h$, $f$, and $g$
are polynomials, without loss of generality they can be expressed in
the form:

\begin{subequations}\label{eq:polyNotation}
\begin{align}
& c(x,u) = \sum_{i=0}^{n_x} \sum_{j=0}^{n_u} c_{ij}x^i u^j,
&& 
h(x) = \sum_{i=0}^{n_x} h_i x^i, \\
& f(x,u) = \sum_{i=0}^{n_x} \sum_{j=0}^{n_u} f_{ij}x^i u^j, 
&& g(x,u) = \sum_{i=0}^{n_x} \sum_{j=0}^{n_u} g_{ij} x^i u^j, \\
& b_l(x,u) = \sum_{i=0}^{n_x} \sum_{j=0}^{n_u} b_{l,ij} x^i u^j.
\end{align}
\end{subequations}

\begin{example}\label{ex:LQR}
A simple example where explicit solutions can be calculated is the
linear quadratic regulator problem. 
For concreteness, a special case
is given by
\begin{subequations}\label{eq:LQROpt}
\begin{align}
& \textrm{minimize} && \left<\int_0^1 \left(\x_t^2 + \u_t^2\right) dt
                       \right>
\\
& \textrm{subject to} && 
d\x_t = \u_t dt + d\w_t \\
&&& \x_0 = 0.
\end{align}
\end{subequations}
See \cite{brysonapplied1975,lewisoptimal2008} for more detailed
discussion of this problem. 
\end{example}

\begin{example}\label{ex:cubicSys}
Another problem, which will display some of the more interesting
aspects of the problem is given by:
\begin{subequations}
\label{eq:cubicOpt}
\begin{align}\label{eq:cubicCost}
& \textrm{minimize} && \left<\int_0^1 \left(\x_t^2 + 0.1 \u_t^2
                       \right) dt
                       + \x_1^2
                       \right> \\
\label{eq:cubicDyn}
& \textrm{subject to} && d\x_t = \left((1.5)^2 \x_t - \x_t^3 + \u_t
                         \right) dt + d\w_t \\
&&& \x_0 = 0.
\end{align}
\end{subequations}
Unlike the linear quadratic regulator problem, it is not clear how to
compute exact solutions to this optimal control problem. 
\end{example}

\begin{example}\label{ex:fisheries}
The previous examples had no constraints constraints of the form
\eqref{eq:polyConstraint}. Such constraints often arise in
applications. For example, consider the modified version of the
optimal fisheries management from
\cite{alvarezoptimal1998,lunguoptimal1997}
\begin{subequations}
\begin{align}
& \textrm{maximize} && \left<
\int_0^T
\u_t
dt
\right> \\
& \textrm{subject to} &&  d\x_t = \left(\x_t-\gamma \x_t^2 -
                         \u_t\right)dt + \sigma \x_t d\w_t \\
&&& \x_0 = x_0 \\
&&& \x_t \ge 0 \\
&&& \u_t \ge 0.
\end{align}
\end{subequations}
Here $\x_t$ models the population in a fishery and $\u_t$ models the
rate of harvesting. As in the earlier works, a constraint that
$\x_t\ge 0$ is required to be physically meaningful. Also, without
this constraint, the optimal strategy would be to set $\u_t =
+\infty$. In other words, the objective would be unbounded without the
constraint. The constraint that $\u_t \ge 0$ encodes the idea that
fish are only being taken out, not put into the fishery. 

The primary difference between this formulation and that of
\cite{alvarezoptimal1998} and \cite{lunguoptimal1997}, is that the
cost is not discounted, and operates over a fixed, finite horizon. 

Note that this is a maximization problem, but this is
equivalent to minimizing the objective multiplied by $-1$.

\end{example}

The systems are assumed to be scalars for notational simplicity. 
Vector systems could be considered at the expense of extra book-keeping. 

\subsection{Description of Results}
\label{sec:description}
Let $v_{\mathrm{OPT}}$ be the optimal value of \eqref{eq:polyCost}. 
For polynomial systems, this paper provides a method based on
semidefinite programming to compute a lower bound on the optimal
value, $v_{\mathrm{SDP}} \le v_{\mathrm{OPT}}$. 

Furthermore, a control policy of the form 
\begin{equation}\label{eq:uCompute}
\u_t = p_0(t) + p_1(t) \x_t + p_2(t) \x_t^2 + \cdots + p_{n_p}(t) \x_t^{n_p}
\end{equation}
can be computed from the result of the semidefinite program. Let $v_p$
denote the value of \eqref{eq:polyCost} resulting from this
controller, which can be computed or estimated by simulations. It follows that 
\begin{equation}
v_{\mathrm{SDP}} \le v_{\mathrm{OPT}} \le v_p.
\end{equation}

While the true optimal value, $v_{\mathrm{OPT}}$ is typically unknown,
if $v_p - v_{\mathrm{SDP}}$ is small, the controller from
\eqref{eq:uCompute} must be close to optimal.

\section{Moment Dynamics of a Controlled SDE}
\label{sec:momentdynamics}
This paper will derive lower bounds for the problem
\eqref{eq:polyProb} by solving an optimal control problem for the
moments of $\x_t$. This section will derive differential equations
for these moments. 

If $q(x)$ is a twice-differentiable,
real-valued function, then the It\^o formula implies that \cite{oksendal03}
\begin{multline}\label{eq:Ito}
dq(\x_t) = \frac{\partial q(\x_t)}{\partial x} \left(
f(\x_t,\u_t) dt + g(\x_t,\u_t) d\w_t
\right)+  \\
\frac{1}{2} \frac{\partial^2 q(\x)}{\partial x^2} g(\x_t,\u_t)^2dt.
\end{multline}

Taking expected values of both sides results in a deterministic
differential equation:
\begin{multline}\label{eq:ExpectationDiffEq}
\frac{d}{dt} \left<q(\x_t)\right> = \\
\left<
\frac{\partial q(\x_t)}{\partial x} 
f(\x_t,\u_t) +  
\frac{1}{2} \frac{\partial^2 q(\x_t)}{\partial x^2} g(\x_t,\u_t)^2
\right>.
\end{multline}

The moments of $\x_t$ and $\u_t$ will be denoted by:
\begin{equation}
\label{eq:momentDef}
\mu_t^{x^iu^j} = \left< \x_t^i \u_t^j \right>.
\end{equation}
When $q(x)$ is a monomial, $q(x) = x^k$, and $f$ and $g$ are
polynomials, \eqref{eq:ExpectationDiffEq} becomes a linear
differential equation with respect to the moments.

\begin{example}\label{ex:cubicMoments}
Recall the dynamics from \eqref{eq:cubicDyn}. Then
\eqref{eq:ExpectationDiffEq} has the form
\begin{multline}
\frac{d}{dt} \left<q(\x_t)\right> 
=
\\
\left<
\frac{\partial q(\x_t)}{\partial x} 
\left((1.5)^2\x_t - \x_t^3 + \u_t\right) +  
\frac{1}{2} \frac{\partial^2 q(\x_t)}{\partial x^2}
\right>.
\end{multline}

For $q(x) = x$, we have that $\frac{\partial q(x)}{\partial x} = 1$
and $\frac{\partial^2 q(x)}{\partial x^2}=0$. Thus, the following holds:
\begin{align}\nonumber
\frac{d}{dt}\left<\x_t\right> &= \frac{d}{dt} \mu_t^{x} \\
\nonumber
&= \left<
(1.5)^2\x_t - \x_t^3 + \u_t 
\right>\\
\label{eq:cubicFirstM}
&= (1.5)^2 \mu_t^x -\mu_t^{x^3} + \mu_t^{u}.
\end{align}
A similar argument shows that 
\begin{align}
\label{eq:cubicSecondM}
\frac{d\mu_t^{x^2}}{dt} &= 2 \left((1.5)^2 \mu_t^{x^2} -\mu_t^{x^4} +
                          \mu_t^{xu}
\right) + 1 \\
\nonumber
\frac{d\mu_t^{x^k}}{dt} &= k\left((1.5)^2 \mu_t^{x^k} -\mu_t^{x^{k+2}} +
                          \mu_t^{x^{k-1}u}
\right) \\
\label{eq:cubicKthM}
& + \frac{k(k-1)}{2} \mu_t^{x^{k-2}} \textrm{ for } k\ge 3.
\end{align}

In this example, we see that the differential equation for $\mu_t^x$
depends on the third moment, $\mu_t^{x^3}$. More generally, the
differential equation for $\mu_t^{x^k}$ depends on the higher moment,
$\mu_t^{x^{k+2}}$. In this case, no moment, $\mu_t^{x^k}$, can be
described using a finite set of differential equations. 
In this case, it is said that the moments are not \emph{closed}. 
Several \emph{moment closure} methods have been developed to
approximate moment dynamics 
using a finite number of differential equations both for discrete, and continuous state Markov models\cite{socha2008linearization, Kuehn16, naasell03,SinghHespanhaLogNormal,SinghHespanhaDM,soltani2015conditional}. 
Future work will involve combining the work in this paper with moment
closure methods. 
\end{example}

\section{Results}
\label{sec:results}

\subsection{Lower Bounds by Semidefinite Programming}

The moment differential equations described in the previous section
must be satisfied for any choice of input strategy for $\u_t$. Thus,
they form a natural candidate for constraints in an optimization
problem.

Another constraint arises from the fact that outer products are
positive semidefinite, and this semidefinite constraint is perserved 
by taking expectations. 
For example:
\begin{multline}
\left<
\begin{bmatrix}
1 \\
\x_t \\
\x_t^2 \\
\u_t
\end{bmatrix}
\begin{bmatrix}
1 &
\x_t &
\x_t^2 &
\u_t
\end{bmatrix}
\right>
= \\
\begin{bmatrix}
1 & \mu_t^x & \mu_t^{x^2} & \mu_t^u \\
\mu_t^x & \mu_t^{x^2} & \mu_t^{x^3} & \mu_t^{xu} \\
\mu_t^{x^2} & \mu_t^{x^3} & \mu_t^{x^4} & \mu_t^{x^2u} \\
\mu_t^{u} & \mu_t^{xu} & \mu_t^{x^2 u} & \mu_t^{u^2}
\end{bmatrix}
\succeq 0.
\end{multline}

Furthermore, if the system has inequality constraints of the from
\eqref{eq:polyConstraint}, then for any $r\ge 1$, it must be the case
that $\left(b_l(\x_t,\u_t)\right)^{r}\ge 0$, and thus 
the following must
hold
\begin{equation}\label{eq:constraintMomentIneq}
\left<
(b_l(\x_t,\u_t))^r
\right> \ge 0.
\end{equation}

Recall the notation for the polynomials from
\eqref{eq:polyNotation}. The following theorem gives a lower bound of
the original problem, \eqref{eq:polyProb}, based on continuous-time
semidefinite programming. 

\begin{theorem}
{\it
For any integers $K\ge 1$, $r_1 \ge 1,\ldots r_B \ge 1$, $d_x \ge 1$, and $d_u\ge 1$, the optimal
value of \eqref{eq:polyProb} is always at least as large as the
optimal value of the following optimal control problem:
\begin{subequations}\label{eq:momentOpt}
\begin{align}
\label{eq:momentObjective}
& \textrm{minimize} && 
\int_0^T\sum_{i=0}^{n_x}\sum_{j=0}^{n_u} c_{ij} \mu_t^{x^i u^j}dt +
\sum_{i=0}^{n_x}h_i \mu_T^{x^i}
\\
\label{eq:momentDyn}
& \textrm{subject to} &&
\frac{d\mu_t^{k}}{dt} = 
k \sum_{i=0}^{n_x}\sum_{j=0}^{n_u} f_{ij} \mu_t^{x^{i+k-1}u^j}
\\ \nonumber
&&&
+\frac{k(k-1)}{2} \sum_{i,r=0}^{n_x}\sum_{j,s=0}^{n_u} g_{ij}g_{rs} 
\mu_t^{x^{i+r+k-2}u^{j+s}} \\
\nonumber
&&& \textrm{ for } k=1,\ldots,K \\
\label{eq:momentIC}
&&& \mu_0^k = x_0^k \quad \textrm{ for } k=1,\ldots,K \\
\nonumber
&&& 
\hspace{-5em}
\sum_{i_1,\ldots,i_{r}}
    \sum_{j_1,\ldots,j_{r}}b_{l,i_1j_1}
    \cdots b_{l,i_{r}j_{r}} \mu_t^{x^{i_1+\cdots+i_{r}}
    u^{j_1+\cdots+j_{r}}}
\ge 0 \\
\label{eq:momentIneq}
&&& \textrm{ for } l=1,\ldots,B \textrm{ and }  r = 1,\ldots,r_l
\\
\label{eq:momentMat}
&&& 
\hspace{-5em}
\left[
\arraycolsep=1.4pt
\begin{array}{ccccccc}
1 & \mu_t^x & \cdots & \mu_t^{x^{d_x}} & \mu_t^u & \cdots &
\mu_t^{u^{d_u}} \\
\mu_t^{x} & \mu_t^{x^2} & \cdots & \mu_t^{x^{d_x+1}}
&
\mu_t^{xu} & \cdots & \mu_t^{xu^{d_u}} \\
\vdots & \vdots & & \vdots & \vdots && \vdots \\ 
\mu_t^{x^{d_x}} & \mu_t^{x^{d_x+1}}
& \cdots & \mu_t^{x^{2d_x}} & \mu_t^{x^{d_x} u}
&\cdots & \mu_t^{x^{d_x} u^{d_u}} \\
\mu_t^u & \mu_t^{xu} & \cdots & \mu_t^{x^{d_x}u} & \mu_t^{u^2} 
&\cdots & \mu_t^{u^{d_u+1}} \\
\vdots & \vdots & & \vdots & \vdots & & \vdots \\
\mu_t^{u^{d_u}}  & \mu_t^{x u^{d_u}} & \cdots &
\mu_t^{x^{d_x}u^{d_u}} & \mu_t^{u^{d_u+1}} & \cdots & \mu_t^{u^{2d_u}}
\end{array}
\right]
\succeq 0.
\end{align}
\end{subequations}
In \eqref{eq:momentIneq}, the sums over $i_1,\ldots,i_{r_l}$ range
from $0$ to $n_x$, while the sums over $j_1,\ldots,j_{r_l}$ range from
$0$ to $n_u$.
}
\end{theorem}
\begin{proof}
The cost function, \eqref{eq:momentObjective}, is exactly, the
original cost, \eqref{eq:polyCost}, expressed in terms of the
moments. Given any policy for $\u_t$, the moments of $\x_t$ must
satisfy \eqref{eq:momentDyn} with initial conditions given by
\eqref{eq:momentIC}. If the system has inequality constraints, as in
\eqref{eq:polyConstraint}, then \eqref{eq:constraintMomentIneq} must
hold. The constraint from \eqref{eq:momentIneq} is exactly
\eqref{eq:constraintMomentIneq} written in terms of the moments. 

The final constraint on the the matrix of moments,
from \eqref{eq:momentMat}, expresses the following semidefiniteness
constraint on the expected value of an outer product:
\begin{equation*}
\left<
\begin{bmatrix}
1\\
\x_t \\
\vdots \\ 
\x_t^{d_x} \\
\u_t \\
\vdots \\
\u_t^{d_u}
\end{bmatrix}
\begin{bmatrix}
1\\
\x_t \\
\vdots \\ 
\x_t^{d_x} \\
\u_t \\
\vdots \\
\u_t^{d_u}
\end{bmatrix}^\tp
\right> \succeq 0.
\end{equation*}
Thus, for every policy, the moments of $\x_t$ and $\u_t$ must satisfy
\eqref{eq:momentMat}. Thus, the constraints \eqref{eq:momentDyn}, \eqref{eq:momentIC},\eqref{eq:momentIneq} and
\eqref{eq:momentMat} are satisfied by every feasible solution for
\eqref{eq:polyProb}. Since their cost functions coincide, \eqref{eq:momentOpt}
is a relaxation of \eqref{eq:polyProb}, and so its optimal solution is
a lower bound to the optimal solution of \eqref{eq:polyProb}. 
\end{proof}

Recall that in some systems, as seen in
Example~\ref{ex:cubicNonClosed}, the moments are not closed. Thus, no
finite number of moment equations will be sufficient to describe the
dynamics exactly. In this case, \eqref{eq:momentOpt} can be used to
construct a sequence of increasing lower bounds by increasing the
number of moment equations, $K$, and increasing the size of the moment
matrix from \eqref{eq:momentMat} correspondingly. The value of the
lower bound increases because the feasible set becomes more
constrained. 

\begin{remark}\label{rem:chooseSize}
There is a large amount of flexibility in choosing the size of the
semidefinite program in \eqref{eq:momentOpt}. One sensible procedure
for choosing the sizes is as follows. Fix $d_x$ and $d_u$. This will
constrain the moment matrix from \eqref{eq:momentMat} to be of size
$(1+d_x+d_u) \times (1+d_x+d_u)$. Then choose the number of moment
differential equations, $K$, and the number of inequality constraints,
$r_1,\ldots, r_B$ to be the largest values such that every moment in
the corresponding constraints is contained in the moment matrix. This
procedure is used in the examples in this paper. 
\end{remark}

As with deterministic continuous-time optimal control problems \cite{raosurvey2009}, the
cost integral, \eqref{eq:momentObjective}, can be discretized as a Riemann
sum, and the dynamic constraints, \eqref{eq:momentDyn}, can be
discretized using Euler integration. This results in a
finite-dimensional semidefinite program. 
Reasonably sized problems, can be handled with off-the-shelf tools for numerical
optimization \cite{diamondcvxpy2016,odonoghueconic2016}. Future work
will developed specialized methods for solving this problem that take
advantage of the specialized structure as an optimal control problem.

\subsection{Constructing the Controller}
\label{sec:construction}

This subsection will describe a procedure for computing a control
strategy of the form in \eqref{eq:uCompute} from a collection of
moments of $\x_t$ and $\u_t$. 

Assume that the $\u_t$ is generated according to
\eqref{eq:uCompute}. Then for every $k\ge 0$, the following holds:
\begin{align}
\nonumber
\mu_t^{x^k u} &= 
\left<
\x_t^k \u_t 
\right> \\
\nonumber
&=
\left<
\x_t^k \left(
p_0(t) + p_1(t) \x_t  + \cdots + p_{n_p}(t) \x_t^{n_p}
\right)
\right> \\
\label{eq:gainConstraint}
&= p_0(t)\mu_t^{x^k} + p_1(t) \mu_t^{x^{k+1}} 
  + 
\cdots + p_{n_p}(t) \mu_t^{x^{k+n_p}}.
\end{align}

Given a collection of moments, as computed from \eqref{eq:momentOpt}, 
coefficients that approximately satisfy \eqref{eq:gainConstraint} may
be computed using least-squares optimization, with the following
objective function:
\begin{equation}\label{eq:pLSTSQ}
\left\|
\arraycolsep=1.4pt
\left[
\begin{array}{cccc}
1 & \mu_t^{x} & \cdots & \mu_t^{x^{n_p}} \\
\mu_t^x & \mu_t^{x^2} & \cdots & \mu_t^{x^{n_p+1}} \\
\vdots & \vdots && \vdots \\
\mu_t^{x^{m}} & \mu_t^{x^{m+1}} & \cdots & 
\mu_t^{x^{m+n_p}}
\end{array}
\right]
\begin{bmatrix}
p_0(t) \\
p_1(t) \\
\vdots \\
p_{n_p}(t)
\end{bmatrix}
-
\begin{bmatrix}
\mu_t^{u} \\
\mu_t^{xu} \\
\vdots \\
\mu_t^{x^{m}u}
\end{bmatrix}
\right\|^2.
\end{equation}
Of course, to pose this optimization problem, the number of moment
differential equations, and the size of the moment matrix from
\eqref{eq:momentOpt} must be large enough so that all of the moments
required for \eqref{eq:pLSTSQ} are computed. 

In practice, the coefficients, $p_i(t)$ are computed at the discrete
time points at which the moments of $\x_t$ and $\u_t$ are computed. 

The controller computed from \eqref{eq:pLSTSQ} is a Markov policy, and
thus feasible. As discussed in Subsection~\ref{sec:description}, the
value of the average cost produced by this controller will always give an
upper bound on the true optimal value.

\section{Examples}
\label{sec:examples}
\begin{example}
Recall the linear quadratic regulator problem from
Example~\ref{ex:LQR}. This problem can be cast in the form of
\eqref{eq:momentOpt} as:

\begin{subequations}
\label{eq:LQRMomentProb}
\begin{align}
& \textrm{minimize} & & \int_0^1 (\mu_t^{x^2} + \mu_t^{u^2} ) dt \\
& \textrm{subject to} && \frac{d\mu_t^x}{dt} = \mu_t^{u} \\
&&& \frac{d\mu_t^{x^2}}{dt} = 2 \mu_t^{xu} + 1 \\
&&& \mu_0^x = \mu_0^{x^2} =  0\\
&&& 
\begin{bmatrix}
1 & \mu_t^x & \mu_t^u \\
\mu_t^x & \mu_t^{x^2} & \mu_t^{xu} \\
\mu_t^u & \mu_t^{xu} & \mu_t^{u^2}
\end{bmatrix}
 \succeq 0.
\end{align}
\end{subequations}
In this example, the moment equations are closed, since moments higher
than $2$ are not required to compute the solution. The classical
theory of optimal control shows that the optimal solution
to the original regulator problem is of the form 
\begin{equation*}
\u_t = L_t \x_t,
\end{equation*}
where $L_t$ is a gain computed from a Riccati differential
equation. Using the Pontryagin maximum principle, it can be shown that
the optimal solution to \eqref{eq:LQRMomentProb} is given by:
\begin{align*}
\begin{bmatrix}
1 & \mu_t^x & \mu_t^u \\
\mu_t^x & \mu_t^{x^2} & \mu_t^{xu} \\
\mu_t^u & \mu_t^{xu} & \mu_t^{u^2}
\end{bmatrix}
 &= 
\begin{bmatrix}
1 & 0 & 0 \\
0 & \mu_t^{x^2} & L_t \mu_t^{x^2} \\
0 & L_t \mu_t^{x^2} & L_t^2 \mu_t^{x^2}
\end{bmatrix} \\
\dot \mu_t^{x^2} &= 2L_t \mu_t^{x^2} + 1, \quad \mu_0^{x^2} = 0.
\end{align*}
\end{example}

\begin{example}\label{ex:cubicNonClosed}
Recall the system from Examples~\ref{ex:cubicSys}
and~\ref{ex:cubicMoments}. The cost function, \eqref{eq:cubicCost} can
be written in moment form as
\begin{equation*}
\int_0^1 \left(
\mu_t^{x^2} + \mu_t^{u^2}
\right)dt + \mu_1^{x^2}. 
\end{equation*}
The problem can be approximated using \eqref{eq:momentOpt} using this
cost, as well as the moment dynamics from \eqref{eq:cubicFirstM},
\eqref{eq:cubicSecondM}, and \eqref{eq:cubicKthM}. Recall that these
moment equations are not closed. Thus, the optimal value can be
approximated by solving \eqref{eq:momentOpt} for increasingly large
numbers of moment equations and increasingly large moment matrices. 

For comparison purposes, Fig.~\ref{fig:cubicCTG} plots the cost-to-go
function,
\begin{equation}\label{eq:cubicCTG}
\int_t^1 \left(\mu_{\tau}^{x^2}+\mu_{\tau}^{u^2}\right) d\tau + \mu_1^{x^2},
\end{equation}
for various sizes of the optimization problem in
\eqref{eq:momentOpt}. For sufficiently large semidefinite programs, an
optimal value of $0.60$ is obtained. 

Using the least-squares method from Subsection~\ref{sec:construction},
an order-$3$ controller was constructed from the solution to the SDP
with $K=4$ moment equations, and moment matrix $d_x=3$ and
$d_u=1$. (This corresponds to a moment matrix of size $(1 + d_x + d_u)
\times  (1 + d_x + d_u) = 5\times 5$.) Running 2000 trials with this
controller resulted in an average cost of $0.67$. Thus, the true
optimal value is likely to lie between $0.60$ and $0.67$. In contrast,
with no control, the average value obtained was $2.4$.

\begin{figure}
\centering
\includegraphics[width=\linewidth]{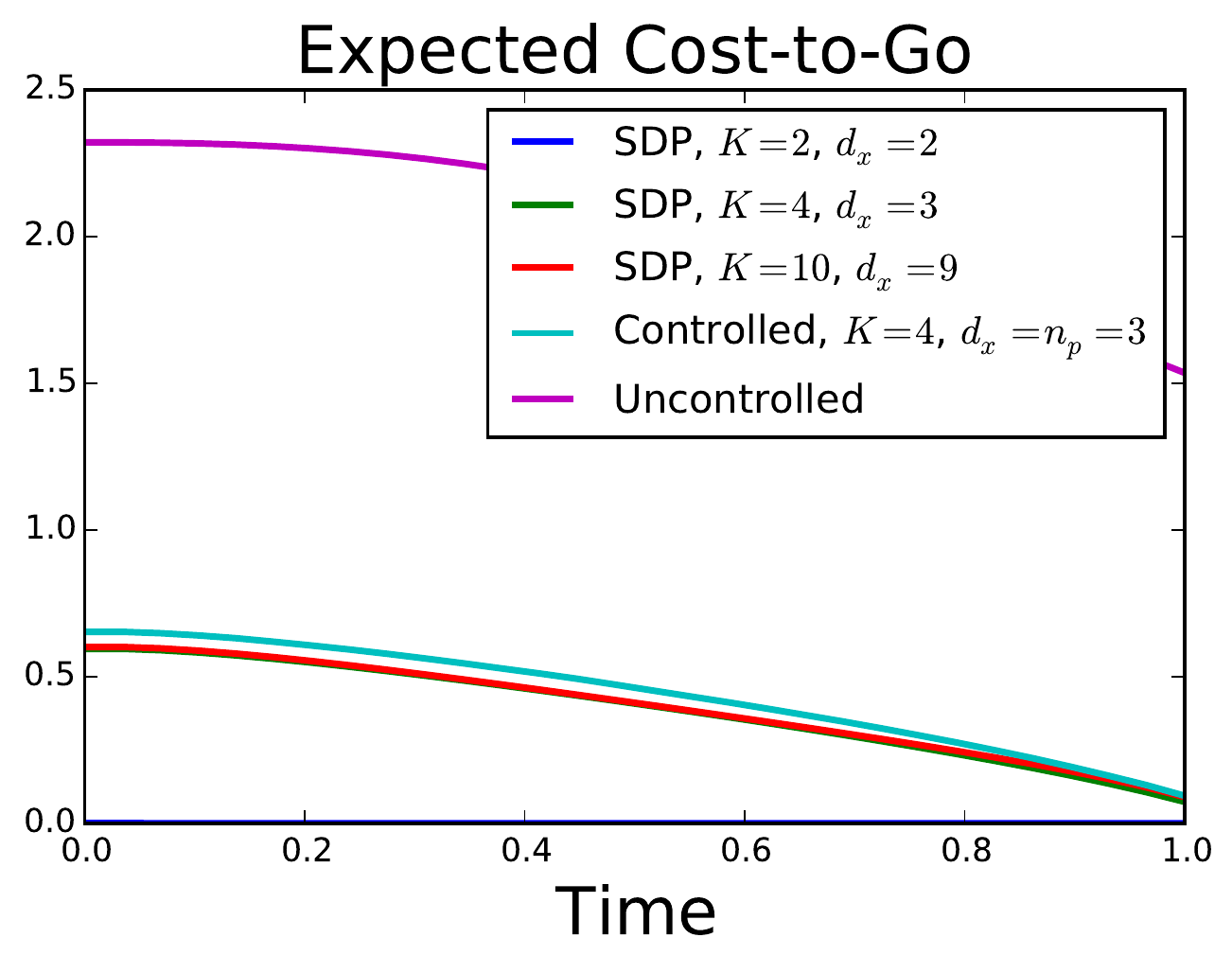}
\caption{\label{fig:cubicCTG} Values of the cost-to-go function,
  \eqref{eq:cubicCTG} are plotted. In all cases, $d_u = 1$, which
  implies the first two moments of $\u_t$ are included in the moment
  matrix. For the smallest SDP, with $K=2$ and $d_x=2$, the moments
  are too unconstrained, and an optimal value of near $0$ is obtained.
  However, with $K=4$ moment equations and $d_x=3$, an optimal value
  of $0.53$ is obtained. Increasing the size to $K=10$ and $d_x=9$
  barely changes this value, or the cost-to-go function. An order-$3$
  controller was constructed from the $K=4$ and $d_x=3$ SDP
  solution. The average cost-to-go of 2000 runs with this controller
  is plotted. The cost-to-go function is quite close to cost-to-go
  from the SDP, aside from a deviation near the end of the
  trajectory. 
  For comparison, the average cost-to-go of 2000 runs of
  the uncontrolled system is also plotted. The average controlled cost
  is $0.66$, compared with an average  uncontrolled cost of $2.4$.}
\end{figure}
\end{example}

\begin{example}

\begin{figure}
\centering
    \begin{minipage}[b]{\linewidth}
    \centering
    \includegraphics[width=\linewidth]{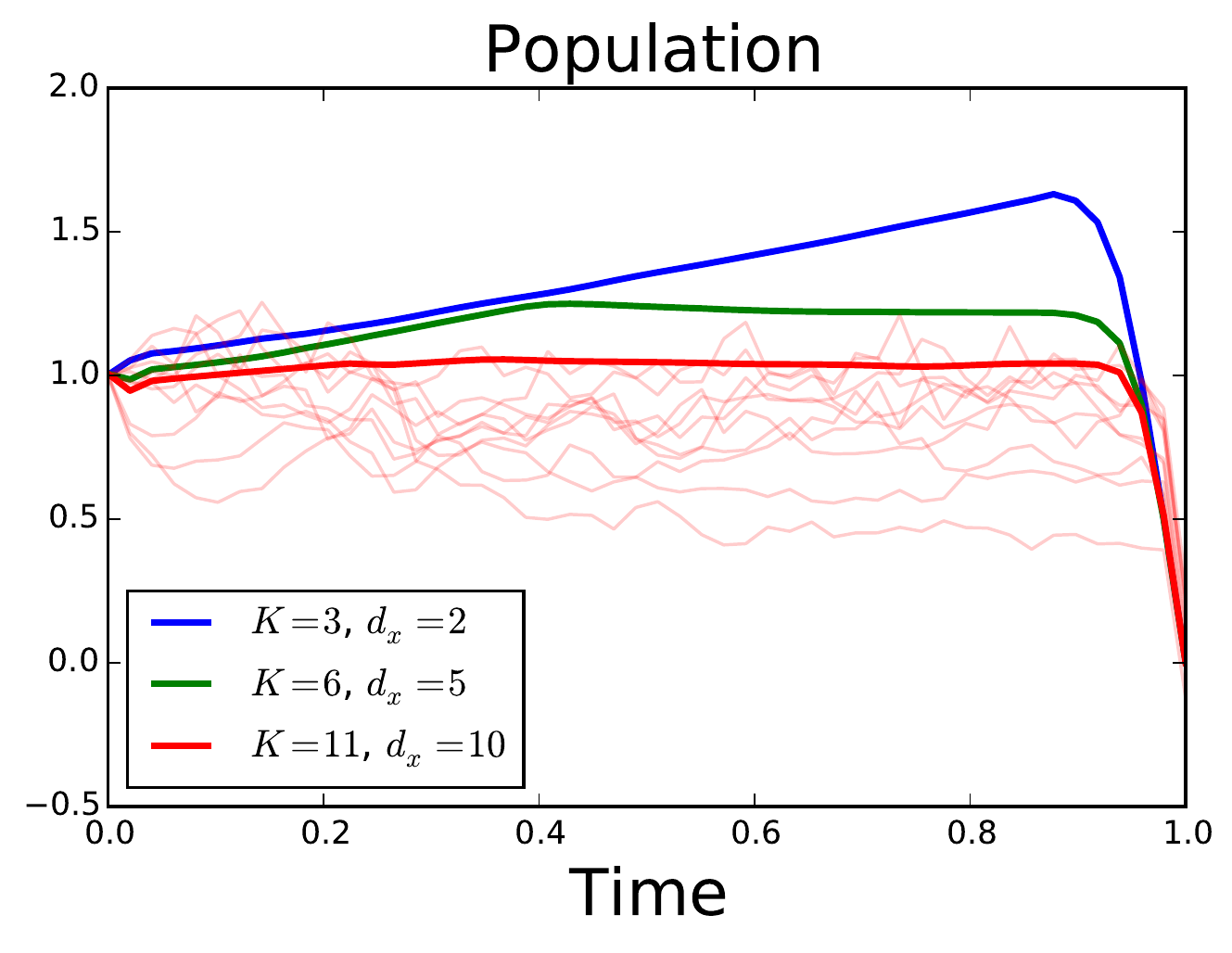}
    \subcaption{\label{fig:population}}
    \end{minipage}
    \begin{minipage}[b]{\linewidth}
    \centering
    \includegraphics[width=\linewidth]{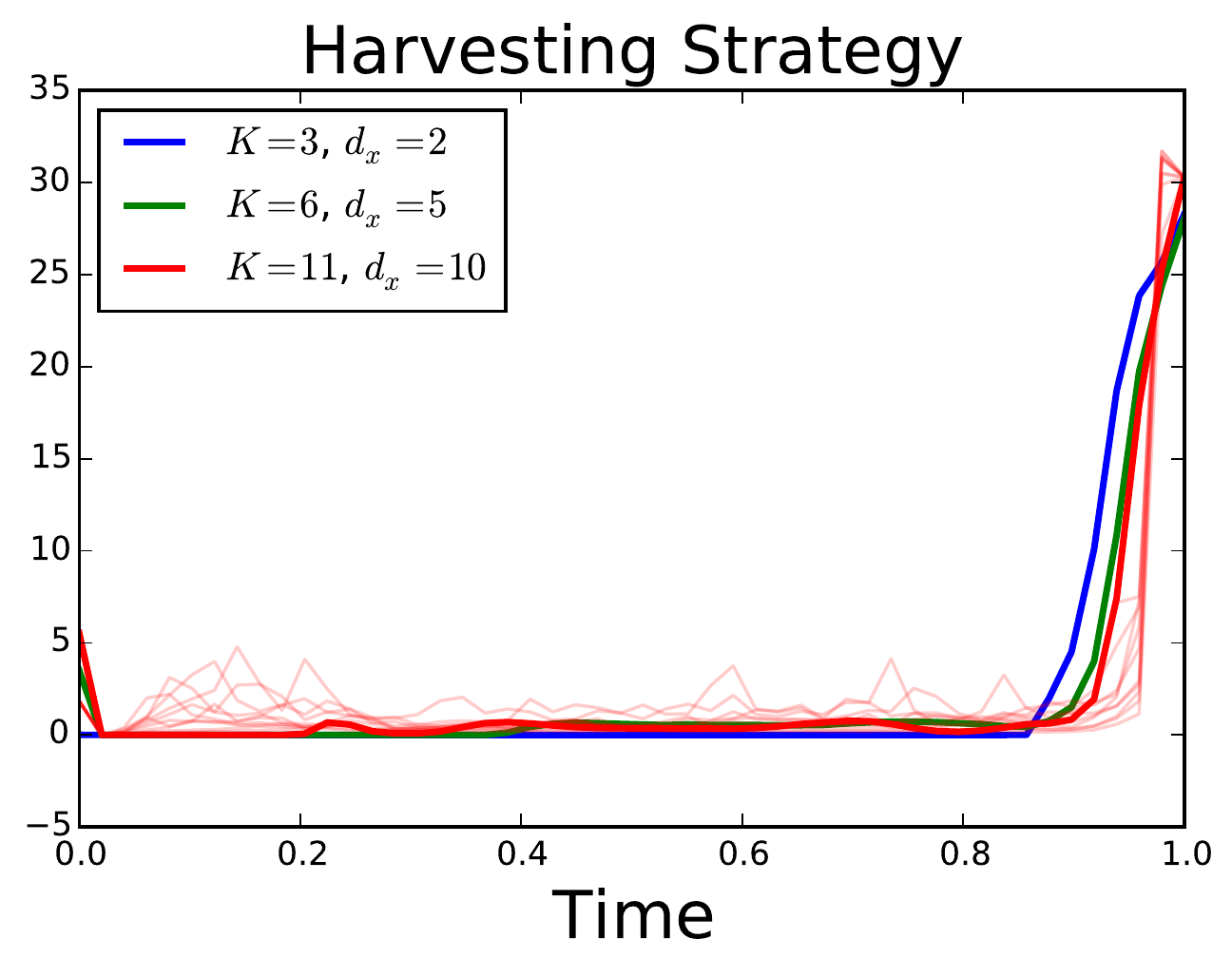}
    \subcaption{\label{fig:harvesting}}
    \end{minipage}
    \caption{ \label{fig:fisheries} \ref{fig:population} The thick
      solid lines show the mean population, $\left<\x_t\right>$
      computed for various sizes of the SDP. In each case $d_u=1$ and
      $d_x$ varies. The sizes of the SDP are
      chosen based on the procedure from
      Remark~\ref{rem:chooseSize}. Using the moments from the $d_x=10$
      SDP, a controller was computed from the procedure from
      Subsection~\ref{sec:construction}. The thin red lines show the
      population $20$ runs of this controller. \ref{fig:harvesting}
      The thick solid lines show mean harvesting rate,
      $\left<\u_t\right>$. Similar to \ref{fig:population}, the thin
      red lines show the harvesting rates in $20$ runs of the computed
      controller. The optimal expected harvest for the various SDPs is
    computed as $2.31$, $2.17$, and $2.11$, for $d_x=2$, $5$, and
    $10$, respectively. Using $1000$ runs of the controlled process
    results in an average harvest of $1.62$. Thus, the true optimal
    harvest is likely to lie somewhere between $1.62$ and $2.11$.}
\end{figure}

Recall Example~\ref{ex:fisheries} on optimal harvesting
for fisheries. 
This problem has moment dynamics given by
\begin{align*}
\dot \mu_t^x &= \mu_t^x -  \gamma \mu_t^{x^2} - \mu_t^u \\
\dot \mu_t^{x^2} &= 2\left(\mu_t^{x^2} - \gamma \mu_t^{x^3} - \mu_t^{xu}\right) + \sigma^2 \mu_t^{x^2} \\
\dot \mu_t^{x^k} &= k\left(\mu_t^{x^k}  - \gamma \mu_t^{x^{k+1}} -
                   \mu_t^{x^{k-1} u}\right) + \frac{1}{2} k(k-1)
                   \mu_t^{x^k}  
\\
&\quad \textrm{for}\quad k \ge 3.
\end{align*}

As with the previous example, the moments are not closed. Furthermore,
this problem has inequality constraints, which imply moment
inequalities,  \eqref{eq:momentIneq}, of the form
\begin{equation*}
\mu_t^{x^r} \ge 0,\qquad \mu_t^{u^r} \ge 0,
\end{equation*}
for $r\ge 1$. 

A plot of the results of the SDP, \eqref{eq:momentOpt}, and a
controller computed from the least squares procedure,
\eqref{eq:pLSTSQ}, is shown in Fig.~\ref{fig:fisheries}. 
From these figures, we can see that a strategy that emerges as the
SDP size increases. Through most of the horizon, harvesting is low and
the population is kept near a constant level. Then, near the end of
the horizon, a large harvest drives the population to $0$. 
\end{example}

\section{Conclusion}
\label{sec:conclusion}
This paper studied a method of solving stochastic optimal control using moment equations. The approach consists of formulating a semidefinite program, with constraints and cost function represented in terms of the moments. Several extensions of this work are possible. For example, it would be interesting to investigate whether by using an appropriate moment-closure for nonlinear systems, we can get a lower bound to the optimal value at low orders of moment truncation. Furthermore, several other works have used moment approximations in conjugation with other techniques for stochastic optimal control \cite{crespo2002stochastic, crespo2003stochastic, xu2012moment,xu2009nonlinear,wojtkiewicz2001moment}. We would like to compare the performance of different controllers, including the computational cost of each.

\section*{Acknowledgments} 
A.S. is supported by the National Science Foundation Grant DMS--1312926.
\bibliography{bibLoc}

\end{document}